\documentclass[12pt,a4paper]{article}
\usepackage{mathrsfs}
\usepackage{epsfig, graphicx}
\usepackage{latexsym,amsfonts,amsbsy,amssymb}
\usepackage{amsmath,amsthm}
\makeatletter % '@' is now a normal "letter" for TeX

\@addtoreset{equation}{section}
\makeatother % '@' is restored as a "non-letter" character for TeX
\allowdisplaybreaks % For long formula
\newtheorem{theorem}{Theorem}[section]

\newtheorem{corollary}{Corollary}[section]
\newtheorem{remark}{Remark}[section]
\newtheorem{algorithm}{Algorithm}[section]

\begin{document}
\title{A Multigrid Method for Nonlinear Eigenvalue Problems: Version 2\footnote{This work is supported in part
by the National Science Foundation of China (NSFC 91330202, 11371026, 11001259,
 11031006, 2011CB309703),  the National
Center for Mathematics and Interdisciplinary Science, CAS and the
President Foundation of AMSS-CAS.}}
\author{ Hehu Xie\footnote{LSEC, ICMSEC,
Academy of Mathematics and Systems Science, Chinese Academy of
Sciences,  Beijing 100190, China (hhxie@lsec.cc.ac.cn)} }
\date{}
\maketitle
\begin{abstract}
A multigrid method is proposed for solving nonlinear
eigenvalue problems by the finite element method. With this new scheme,
solving nonlinear eigenvalue problem is decomposed to a series of solutions of linear boundary value
problems on multilevel finite element spaces and a series of small scale nonlinear eigenvalue problems.
The computational work of this new scheme can reach almost the same
as the solution of the corresponding linear boundary value problem. Therefore, this type
of multilevel correction scheme improves the overfull efficiency of the nonlinear eigenvalue problem solving.


\end{abstract}
%-------------------------------------------------------------------------
\section{Introduction}
It is well know that solving large scale eigenvalue problems becomes a fundamental problem in modern science and engineering society.
Among these eigenvalue problems, there exist many nonlinear eigenvalue problems
\cite{Bao,BaoDu,CancesChakirMaday,ChenGongHeYangZhou,ChenGongZhou,ChenHeZhou,KohnSham,Martin,ParrYang,SulemSulem}.
However, it is not an easy task to solve high-dimensional nonlinear eigenvalue problems which come from
physical and chemical sciences.

The multigrid method and other efficient preconditioners provide an optimal order algorithm for solving
boundary value problems since they can obtain the theoretical error by the linear scale computation work.
% The error bounds of the approximate solution
%obtained from these efficient numerical algorithms are comparable to the theoretical
%bounds determined by the finite element discretization.  But the amount of computational
%work involved is only proportional to the number of unknowns
%in the discretized equations. For more details of the multigrid and multilevel methods,
%please read 
We introduce the papers:  %Bank and Dupont \cite{BankDupont},  Bramble and Pasciak \cite{BramblePasciak},
Bramble and Zhang \cite{BrambleZhang},  Scott and Zhang \cite{ScottZhang}, Xu \cite{Xu},
and books:  Bramble \cite{Bramble}, Brenner and Scott \cite{BrennerScott},  Hackbusch \cite{Hackbusch_Book},
McCormick \cite{McCormick}, Shaidurov \cite{Shaidurov} to the interested readers.

Recently, we develop a type of
multigrid method for linear eigenvalue problems \cite{LinXie,LinXie_Multigrid,Xie_Steklov,Xie_Nonconforming,Xie_JCP}.
Then the aim of this paper is to present a type of multigrid
scheme  for nonlinear
eigenvalue problems based on the multilevel correction method \cite{LinXie}.
 With this method, solving nonlinear eigenvalue problem will not be
more difficult than solving the corresponding linear boundary value problem.
The multigrid method for nonlinear eigenvalue problem is based on a
series of finite element spaces with different level of accuracy
 which can be built with the same way as the multilevel
 method for boundary value problems \cite{Xu}.
It is worth pointing out that besides the multigrid method,
 other types of numerical algorithms such as BPX multilevel preconditioners,
algebraic multigrid method and domain decomposition preconditioners \cite{BrennerScott}
 can also act as the linear algebraic solvers for the
 multigrid method of the nonlinear eigenvalue problem.

The corresponding error and computational work estimates of the proposed multigrid
 scheme for the nonlinear eigenvalue problem will be analyzed. Based
on the analysis, the new method can obtain optimal errors with an almost optimal
computational work. The eigenvalue
multigrid procedure can be described as follows: (1)\
solve the nonlinear eigenvalue problem in the coarsest finite element space;
(2)\ solve an additional linear boundary value problem with multigrid method on the refined mesh using
the previous obtained eigenvalue multiplying the corresponding
eigenfunction as the load vector; (3)\ solve a nonlinear eigenvalue problem
again on the finite element space which is constructed by combining
the coarsest finite element space with the obtained eigenfunction
approximation in step (2). Then go to step (2) for next loop until stop.
In this method, we replace solving nonlinear eigenvalue problem on the finest
finite element space by solving a series of linear boundary value problems with multigrid scheme
in the corresponding series of finite element spaces and a series of nonlinear
eigenvalue problems in the coarsest finite element space. So this multigrid method
can improve the overfull efficiency of solving eigenvalue problems.
% as it does
%for linear boundary value problems (with the computational work $\mathcal{O}(N_n)$).

An outline of the paper goes as follows. In Section 2, we introduce
finite element method for nonlinear eigenvalue problem and some
 assumptions in this paper. Two correction steps are given in Sections
 3 and 4 based on fixed-point iteration and Newton iteration, respectively.
 In Section 5, we propose a type of multigrid
algorithm for solving the nonlinear eigenvalue problem by finite element method.
Section 6 is devoted to estimating the computational work for the multigrid method defined in Section 5.
Some concluding remarks are given in the last section.

\section{Finite element method for nonlinear eigenvalue problem}
In this section, we introduce the finite element method for the nonlinear
eigenvalue problem, some notation and error estimates of
the finite element approximation for eigenvalue problems.
The letter $C$ (with or without subscripts) denotes a generic
positive constant which may be different at its different occurrences through the paper.
For convenience, the symbols $\lesssim$, $\gtrsim$ and $\approx$
will be used in this paper. That $x_1\lesssim y_1, x_2\gtrsim y_2$
and $x_3\approx y_3$, mean that $x_1\leq C_1y_1$, $x_2 \geq c_2y_2$
and $c_3x_3\leq y_3\leq C_3x_3$ for some constants $C_1, c_2, c_3$
and $C_3$ that are independent of mesh sizes (see, e.g., \cite{Xu}).
We use the standard notation for Sobolev spaces $W^{s,p}(\Omega)$
and their associated norms, semi-norms
\cite{BrennerScott,Ciarlet}. For $p=2$, denote $H^s(\Omega)=W^{s,2}(\Omega)$
and $H_0^1(\Omega)=\{v\in H^1(\Omega): v|_{\partial\Omega}=0\}$, where
$v|_{\partial\Omega}$ is understand in the sense of trace, $\|\cdot\|_{s,\Omega}=\|\cdot\|_{s,2,\Omega}$,
and $(\cdot,\cdot)$ is the standard $L^2(\Omega)$ inner product.

In this paper, we are concerned with the following nonlinear eigenvalue problem:\\
Find $(\lambda,u)$ such that
\begin{equation}\label{Nonlinear_Eigenvalue_Problem}
\left\{
\begin{array}{rcl}
-\Delta u+f(x,u)&=&\lambda u,\ \ \ {\rm in}\ \Omega,\\
u&=&0,\ \ \ \ \ {\rm on}\ \partial\Omega,\\
\int_{\Omega}u^2d\Omega&=&1,
\end{array}
\right.
\end{equation}
where $\Omega\subset \mathcal{R}^d$ denotes the computing domain and $f(x,u)$
is a smooth enough function such that the eigenvalue problem (\ref{Nonlinear_Eigenvalue_Problem})
has only real eigenvalues.

In this paper, we set $V=H_0^1(\Omega)$.
For the aim of finite element discretization, we define the corresponding
weak eigenvalue problem as follows:\\
Find $(\lambda,u)\in \mathcal{R}\times V$ such that $b(u,u)=1$ and
\begin{eqnarray}
a(u,v)&=&\lambda b(u,v),\quad \forall v\in V, \label{weak_problem}
\end{eqnarray}
where
\begin{eqnarray*}
a(u,v):=\int_{\Omega}\big(\nabla u\nabla v+f(x,u)v\big)d\Omega,\ \ \
b(u,v):=\int_{\Omega}uvd\Omega.
\end{eqnarray*}

Now, let us define the finite element approximations of the problem
(\ref{weak_problem}). First we generate a shape-regular
decomposition of the computing domain $\Omega\subset \mathcal{R}^d\
(d=2,3)$ into triangles or rectangles for $d=2$ (tetrahedrons or
hexahedrons for $d=3$). The diameter of a cell $K\in\mathcal{T}_h$
is denoted by $h_K$. The mesh diameter $h$ describes the maximum
diameter of all cells $K\in\mathcal{T}_h$. Based on the mesh
$\mathcal{T}_h$, we can construct the linear finite element space denoted by
$V_h\subset V$. In order to apply multigrid scheme, we
start the process on the original mesh $\mathcal{T}_H$ with the mesh
size $H$ and the original coarse linear finite element space $V_H$
defined on the mesh $\mathcal{T}_H$.  We assume that
$V_h\subset V$ is a family of finite-dimensional spaces that satisfy
the following assumption:\\
For any $w \in V$
\begin{eqnarray}\label{Approximation_Property}
\lim_{h\rightarrow0}\inf_{v\in V_h}\|w-v\|_1 = 0.
\end{eqnarray}

The standard finite element method is to solve the following eigenvalue problem:\\
Find $(\bar\lambda_h, \bar u_h)\in \mathcal{R}\times V_h$ such that
$b(\bar u_h,\bar u_h)=1$ and
\begin{eqnarray}\label{weak_problem_Discrete}
a(\bar u_h,v_h)&=&\bar\lambda_hb(\bar u_h,v_h),\quad\ \  \ \forall v_h\in V_h.
\end{eqnarray}
Then we define
\begin{eqnarray}
\delta_h(u)=\inf_{v_h\in V_h}\|u-v_h\|_1.
\end{eqnarray}

For generality, we only state the following assumptions about the error estimate for the eigenpair
approximation $(\bar\lambda_h,\bar u_h)$ defined by (\ref{weak_problem_Discrete}) (see, e.g., \cite{CancesChakirMaday,ChenHeZhou}).

{\bf Assumption A1}:
The eigenpair approximation $(\bar\lambda_h,\bar u_h)$ of (\ref{weak_problem_Discrete}) has the following
error estimates
\begin{eqnarray}
\|u-\bar u_h\|_1 &\lesssim& \delta_h(u),\label{Error_Estimate_Eigenfunction}\\
|\lambda-\bar\lambda_h|+\|u-\bar u_h\|_0&\lesssim & \eta_a(V_h)\|u-\bar u_h\|_1,\label{Error_Estimate_Eigenvalue}
\end{eqnarray}
where $\eta_a(V_h)$ depends on the finite dimensional space $V_h$ and has the following property
\begin{eqnarray}\label{Property_Eta_h}
\lim_{h\rightarrow 0}\eta_a(V_h)=0, \ \ \ \eta_a(\widetilde{V}_{h})\leq \eta_a(V_h)\ \ {\rm if}\
V_h\subset \widetilde{V}_h\subset V.
\end{eqnarray}

{\bf Assumption A2}:\ Assume $V^h$ is a subspace of $V_h$.
Let us define the eigenpair approximation $(\lambda^h,u^h)$ by solving the eigenvalue problem as follows:

Find $(\lambda^h,u^h)\in\mathcal{R}\times V^h$ such that $b(u^h,u^h)=$ and
\begin{eqnarray}\label{Nonlinear_Eigenvalue_Problem_subspace}
a(u^h,v^h)&=&\lambda^hb(u^h,v^h),\ \ \ \ \forall v^h\in V^h.
\end{eqnarray}
Then the following error estimates hold
\begin{eqnarray}
\|\bar{u}_h-u^h\|_1 &\lesssim& \delta_h(\bar{u}_h),\label{Error_Estimate_Eigenfunction_Subspace}\\
|\bar{\lambda}_h-\lambda^h|+\|\bar{u}_h-u^h\|_0&\lesssim & \eta_a(V^h)\|\bar{u}_h-u^h\|_1,
\label{Error_Estimate_Eigenvalue_Subspace}
\end{eqnarray}
where
\begin{eqnarray}\label{Detlat_Definition_Subspace}
\delta_h(\bar{u}_h):=\inf_{v^h\in V^h}\|\bar{u}_h-v^h\|_1.
\end{eqnarray}

In order to design and analyze the multilevel correction method for the nonlinear eigenvalue problems, we also
need the following assumptions for the nonlinear function $f(\cdot,\cdot):\mathcal{R}\times V\rightarrow \mathcal{R}$.

%The eigenpair $(\lambda,u)$ defined by (\ref{weak_problem}) and its
%approximation $(\lambda_h,u_h)$ by (\ref{weak_problem_Discrete}) have the following estimate
%for the function $f(\cdot,\cdot):\mathcal{R}\times V\rightarrow \mathcal{R}$

{\bf Assumption B}:
The nonlinear function $f(x,\cdot)$ has the following estimate
\begin{eqnarray}\label{Nonlinear_Estimate_Fix}
|(f(x,w)-f(x,v),\psi)|\lesssim \|w-v\|_0\|\psi\|_1,\ \ \ \forall w\in V, \ \ \forall v\in V,\ \ \forall \psi\in V.
\end{eqnarray}

{\bf Assumption C}:
The nonlinear function $f(x,\cdot)$ has the following estimate
\begin{eqnarray}\label{Nonlinear_Estimate_Newton}
|(f(x,w)-f(x,v)-f_v(x,v)(w-v),\psi)|&\lesssim& \|w-v\|_0\|\psi\|_1,\ \ \forall w\in V,\nonumber\\
&&\ \ \forall v\in V,\ \ \forall \psi\in V.
\end{eqnarray}
For more discussions about the function $f(x,\cdot)$, please refer to
\cite{CancesChakirMaday,ChenGongHeYangZhou,Xu_Nonlinear} and the papers cited therein.

\section{One correction step based on fixed-point iteration}
In this section, we introduce a type of correction step based on the fixed-point iteration
 to improve the accuracy of the current eigenpair approximation.
This correction step contains solving an auxiliary linear boundary value problem with multigrid method
in the finer finite element space and a nonlinear eigenvalue problem on the
coarsest finite element space.

%For simplicity of notation, we set
%$(\lambda,u)=(\lambda_i,u_i)\ (i=1,2,\cdots,k,\cdots)$  and
%$(\lambda_h, u_h)=(\lambda_{i,h},u_{i,h})\ (i=1,2,\cdots,N_h)$ to
%denote an eigenpair and the corresponding approximation of problem (\ref{weak_problem}) and
%(\ref{weak_problem_Discrete}), respectively.
Assume we have obtained an eigenpair approximation
$(\lambda_{h_k},u_{h_k})\in\mathcal{R}\times V_{h_k}$. Now we
introduce a type of correction step to improve the accuracy of the
current eigenpair approximation $(\lambda_{h_k},u_{h_k})$. Let
$V_{h_{k+1}}\subset V$ be a finer finite element space such that
$V_{h_k}\subset V_{h_{k+1}}$. Based on this finer finite element space,
we define the following correction step.

\begin{algorithm}\label{Correction_Step_Fix}
One Correction Step based on Fixed-point Iteration

\begin{enumerate}
\item Define the following auxiliary boundary value problem:

Find $\widehat{u}_{h_{k+1}}\in V_{h_{k+1}}$ such that
\begin{eqnarray}\label{aux_problem_fix}
\hskip-0.5cm (\nabla\widehat{u}_{h_{k+1}},\nabla v_{h_{k+1}})=
\lambda_{h_k}b(u_{h_k},v_{h_{k+1}})-(f(x,u_{h_k}),v_{h_{k+1}}), \ \forall v_{h_{k+1}}\in V_{h_{k+1}}.
\end{eqnarray}
Solve this equation with multigrid method to obtain an approximation
$\widetilde{u}_{h_{k+1}}\in V_{h_{k+1}}$ with error estimate
\begin{eqnarray}\label{Multigrid_Accuracy}
\|\widehat{u}_{h_{k+1}}-\widetilde{u}_{h_{k+1}}\|_a\leq C\eta_a(V_{h_k})\delta_{h_k}(u).
\end{eqnarray}
% and set $\widetilde{u}_{h_{k+1}}=u_{h_k}+\widetilde{e}_{h_{k+1}}$.
% and define $\widetilde{e}_{h_{k+1}}:=MG(V_{h_{k+1}},e_{h_k}, \lambda_{h_k},u_{h_k},m_2)$, where $V_{h_{k+1}}$ denotes the
% computing space, $u_{h_k}$ is the initial solution, $\lambda_{h_k}u_{h_k}$ the right hand side and $m_2$ the
% iteration time of the multigrid scheme.
\item  Define a new finite element
space $V_{H,h_{k+1}}=V_H+{\rm span}\{\widetilde{u}_{h_{k+1}}\}$ and solve
the following eigenvalue problem:

Find $(\lambda_{h_{k+1}},u_{h_{k+1}})\in\mathcal{R}\times V_{H,h_{k+1}}$ such
that $b(u_{h_{k+1}},u_{h_{k+1}})=1$ and
\begin{eqnarray}\label{Eigen_Augment_Problem_fix}
a(u_{h_{k+1}},v_{H,h_{k+1}})&=&\lambda_{h_{k+1}} b(u_{h_{k+1}},v_{H,h_{k+1}}),\ \ \
\forall v_{H,h_{k+1}}\in V_{H,h_{k+1}}.
\end{eqnarray}
\end{enumerate}
Summarize above two steps into
\begin{eqnarray*}
(\lambda_{h_{k+1}},u_{h_{k+1}})={\it
Correction}(V_H,\lambda_{h_k},u_{h_k},V_{h_{k+1}}).
\end{eqnarray*}
%where $V_H$ denotes the coarsest finite element space, $\lambda_{h_k}$ and
%$u_{h_k}$ are the given eigenvalue and eigenfunction approximations, respectively, $V_{h_{k+1}}$ denotes the
% computing space.
\end{algorithm}
%--------------------------------------------------------------------------------------
\begin{theorem}\label{Error_Estimate_One_Correction_Theorem_Fix}
Assume {\bf Assumptions A1}, {\bf A2} and {\bf B} hold.
The resultant approximation $(\lambda_{h_{k+1}},u_{h_{k+1}})\in\mathcal{R}\times V_{h_{k+1}}$ by Algorithm \ref{Correction_Step_Fix}
and the  eigenpair approximation $(\bar{\lambda}_{h_{k+1}},\bar{u}_{h_{k+1}})$ by the direct finite element method
in $V_{h_{k+1}}$ have the following estimates
\begin{eqnarray}
\|\bar{u}_{h_{k+1}}-u_{h_{k+1}}\|_1  &\lesssim& \varepsilon_{h_{k+1}}(u),
\label{Estimate_u_u_h_{k+1}_fix}\\
|\bar{\lambda}_{h_{k+1}}-\lambda_{h_{k+1}}|+\|\bar u_{h_{k+1}}-u_{h_{k+1}}\|_0
&\lesssim&\eta_a(V_H)\|\bar{u}_{h_{k+1}}-u_{h_{k+1}}\|_1,
\label{Estimate_u_h_{k+1}_Nagative_fix}\\
|(f(x,\bar{u}_{h_{k+1}})-f(x,u_{h_{k+1}}),v)|
&\lesssim&\eta_a(V_H)\|\bar{u}_{h_{k+1}}-u_{h_{k+1}}\|_1\|v\|_1,
 \ \forall v\in V.\label{Nonlinear_Estimate_k+1_fix}
\end{eqnarray}
where $\varepsilon_{h_{k+1}}(u):= \eta_a(V_{h_k})\delta_{h_k}(u)
+\|\bar{u}_{h_k}-u_{h_k}\|_0+|\bar{u}_{h_k}-\lambda_{h_k}|$.
\end{theorem}
\begin{proof}
From (\ref{weak_problem_Discrete}) and (\ref{aux_problem_fix}),
the following inequalities hold for any $v_{h_{k+1}}\in V_{h_{k+1}}$
\begin{eqnarray*}
&&\big(\nabla(\bar{u}_{h_{k+1}}-\widehat{u}_{h_{k+1}}),\nabla v_{h_{k+1}}\big)\nonumber\\
&=&b(\bar{\lambda}_{h_{k+1}}\bar{u}_{h_{k+1}}-\lambda_{h_k}u_{h_k},v_{h_{k+1}})+
(f(x,u_{h_k})-f(x,\bar{u}_{h_{k}}),v_{h_{k+1}})\nonumber\\
&\lesssim&\big(|\bar{\lambda}_{h_{k+1}}-\lambda_{h_k}|
+\|\bar{u}_{h_{k+1}}-u_{h_k}\|_0\big)\|v_{h_{k+1}}\|_1\nonumber\\
&\lesssim&\big(|\bar{\lambda}_{h_{k+1}}-\bar{\lambda}_{h_k}|+|\bar{\lambda}_{h_k}-\lambda_{h_k}|
+\|\bar{u}_{h_{k+1}}-\bar{u}_{h_k}\|_0+\|\bar{u}_{h_k}-u_{h_k}\|_0\big)\|v_{h_{k+1}}\|_1\nonumber\\
&\lesssim& \big(\eta_a(V_{h_k})\delta_{h_k}(u)
+\|\bar{u}_{h_k}-u_{h_k}\|_0+|\bar{u}_{h_k}-\lambda_{h_k}|\big)\|v_{h_{k+1}}\|_1.
\end{eqnarray*}
Then we have
\begin{eqnarray}\label{Estimate_u_tilde_u_h_{k+1}_fix}
\|\bar{u}_{h_{k+1}}-\widehat{u}_{h_{k+1}}\|_1
\lesssim \eta_a(V_{h_k})\delta_{h_k}(u)+\|\bar{u}_{h_k}-u_{h_k}\|_0+|\bar{u}_{h_k}-\lambda_{h_k}|.
\end{eqnarray}
Combining (\ref{Estimate_u_tilde_u_h_{k+1}_fix}) and the accuracy (\ref{Multigrid_Accuracy}) leads to the
following estimate
\begin{eqnarray}\label{Error_tilde_u_h_{k+1}_u_final_fix}
\|\bar{u}_{h_{k+1}}-\widetilde{u}_{h_{k+1}}\|_1\lesssim
\eta_a(V_{h_k})\delta_{h_k}(u)+\|\bar{u}_{h_k}-u_{h_k}\|_0+|\bar{u}_{h_k}-\lambda_{h_k}|.
\end{eqnarray}
Now we come to estimate the error for the eigenpair solution
$(\lambda_{h_{k+1}},u_{h_{k+1}})$ of problem (\ref{Eigen_Augment_Problem_fix}).
Based on  {\bf Assumptions A1}, {\bf A2} and {\bf B}, and the definition of $V_{H,h_{k+1}}$,
the following estimates hold
\begin{eqnarray}\label{Error_u_u_h_{k+1}_fix}
\|\bar{u}_{h_{k+1}}-u_{h_{k+1}}\|_1\lesssim \inf_{v_{H,h_{k+1}}\in
V_{H,h_{k+1}}}\|\bar{u}_{h_{k+1}}-v_{H,h_{k+1}}\|_1
\lesssim \|\bar{u}_{h_{k+1}}-\widetilde{u}_{h_{k+1}}\|_1,
\end{eqnarray}
and
\begin{eqnarray}
|\bar{\lambda}_{h_{k+1}}-\lambda_{h_{k+1}}|+\|\bar{u}_{h_{k+1}}-u_{h_{k+1}}\|_0
&\lesssim& \eta_a(V_{H,h_{k+1}})\|\bar{u}_{h_{k+1}}-u_{h_{k+1}}\|_1,\label{Error_u_u_h_{k+1}_Negative_Fix}\\
|(f(x,\bar{u}_{h_{k+1}})-f(x,u_{h_{k+1}}),v)| &\lesssim&
\eta_a(V_{H,h_{k+1}})\|\bar{u}_{h_{k+1}}-u_{h_{k+1}}\|_1\|v\|_1,\nonumber\\
&&\ \ \ \ \ \ \quad\quad\quad\quad\quad\quad\  \forall v\in V.\label{Error_Nonlinear_k+1_Fix}
\end{eqnarray}
From (\ref{Property_Eta_h}), (\ref{Error_tilde_u_h_{k+1}_u_final_fix}), (\ref{Error_u_u_h_{k+1}_fix}),
(\ref{Error_u_u_h_{k+1}_Negative_Fix}) and (\ref{Error_Nonlinear_k+1_Fix}), we can obtain the desired results
(\ref{Estimate_u_u_h_{k+1}_fix}), (\ref{Estimate_u_h_{k+1}_Nagative_fix}) and
(\ref{Nonlinear_Estimate_k+1_fix}).
\end{proof}

\section{One correction step based on Newton iteration}
In this section, we present another type of correction step based on Newton iteration
(always has better convergence property)
to improve the accuracy of the given eigenpair approximations.
This correction method also contains solving an auxiliary linear
 boundary value problem with multigrid method
in the finer finite element space and a nonlinear eigenvalue problem on the
coarsest finite element space.

Similarly, assume we have obtained an eigenpair approximation
$(\lambda_{h_k},u_{h_k})\in\mathcal{R}\times V_{h_k}$.
Let $V_{h_{k+1}}\subset V$ be a finer finite element space such that
$V_{h_k}\subset V_{h_{k+1}}$.

In this section, we define the bilinear form $a_{h_k}(w,v)$ as follows
\begin{eqnarray}\label{Definition_a_h_k}
a_{h_k}(w,v)=(\nabla w,\nabla v)+(f_u(x,u_{h_k})w,v).
\end{eqnarray}
Here, we assume the linearized operator $L_u:=-\Delta +f_u(x,u)$ is nonsingular and
 $u_{h_k}$ is close enough to $u$
such that the following properties hold \cite[Lemma 2.1]{Xu_Nonlinear}
\begin{eqnarray}
\sup_{0\neq v_{h_{k+1}}\in V_{h_{k+1}}}\frac{a_{h_k}(w_{h_{k+1}},v_{h_{k+1}})}{\|v_{h_{k+1}}\|_1}
&\gtrsim& \|w_{h_{k+1}}\|_1,\ \ \ \ \ \forall w_{h_{k+1}}\in V_{h_{k+1}},\label{Inf_Sup_Condition}\\
|a_{h_k}(w,v)|&\lesssim& \|w\|_1\|v\|_1,\ \ \ \ \ \forall w\in V,\ \forall v\in V.\label{Boundednness}
\end{eqnarray}

Now we define the correction step as follows.

\begin{algorithm}\label{Correction_Step_Newton}
One Correction Step based on Newton Iteration

\begin{enumerate}
\item Define the following auxiliary boundary value problem:

Find $\widehat{e}_{h_{k+1}}\in V_{h_{k+1}}$ such that
\begin{eqnarray}\label{aux_problem}
a_{h_k}(\widehat{e}_{h_{k+1}},v_{h_{k+1}})&=&\lambda_{h_k}b(u_{h_k},v_{h_{k+1}})-(\nabla u_{h_k},\nabla v_{h_{k+1}})
,\nonumber\\
&&\quad\quad\ -(f(x,u_{h_k}),v_{h_{k+1}}), \ \ \forall v_{h_{k+1}}\in V_{h_{k+1}}.
\end{eqnarray}
Solve this equation with multigrid method \cite{Shaidurov,Xu_Two_Grid} to obtain an approximation
$\widetilde{e}_{h_{k+1}}\in V_{h_{k+1}}$ with error estimate
$\|\widehat{e}_{h_{k+1}}-\widetilde{e}_{h_{k+1}}\|_a\leq C\delta_{h_{k+1}}(u)$
 and set $\widetilde{u}_{h_{k+1}}=u_{h_k}+\widetilde{e}_{h_{k+1}}$.
\item  Define a new finite element
space $V_{H,h_{k+1}}=V_H+{\rm span}\{\widetilde{u}_{h_{k+1}}\}$ and solve
the following eigenvalue problem:

Find $(\lambda_{h_{k+1}},u_{h_{k+1}})\in\mathcal{R}\times V_{H,h_{k+1}}$ such
that $b(u_{h_{k+1}},u_{h_{k+1}})=1$ and
\begin{eqnarray}\label{Eigen_Augment_Problem}
a(u_{h_{k+1}},v_{H,h_{k+1}})&=&\lambda_{h_{k+1}} b(u_{h_{k+1}},v_{H,h_{k+1}}),\ \ \
\forall v_{H,h_{k+1}}\in V_{H,h_{k+1}}.
\end{eqnarray}
\end{enumerate}
Summarize above two steps onto
\begin{eqnarray*}
(\lambda_{h_{k+1}},u_{h_{k+1}})={\it
Correction}(V_H,\lambda_{h_k},u_{h_k},V_{h_{k+1}}).
\end{eqnarray*}
%where $V_H$ denotes the coarsest finite element space, $\lambda_{h_k}$ and
%$u_{h_k}$ are the given eigenvalue and eigenfunction approximations, respectively, $V_{h_{k+1}}$ denotes the
% computing space.
\end{algorithm}
%--------------------------------------------------------------------------------------
\begin{theorem}\label{Error_Estimate_One_Correction_Theorem}
Assume {\bf Assumptions A1}, {\bf A2} and {\bf C} hold.
The resultant approximation $(\lambda_{h_{k+1}},u_{h_{k+1}})\in\mathcal{R}\times V_{h_{k+1}}$ by Algorithm \ref{Correction_Step_Newton}
and the eigenpair approximation $(\bar{\lambda}_{h_{k+1}},\bar{u}_{h_{k+1}})$ by the direct finite element method
in $V_{h_{k+1}}$ have the following estimates
\begin{eqnarray}
&&\|\bar{u}_{h_{k+1}}-u_{h_{k+1}}\|_1\lesssim \varepsilon_{h_{k+1}}(u),\label{Estimate_u_u_h_{k+1}}\\
&&|\bar{\lambda}_{h_{k+1}}-\lambda_{h_{k+1}}|+\|\bar{u}_{h_{k+1}}-u_{h_{k+1}}\|_0
\lesssim\eta_a(V_H)\|\bar{u}_{h_{k+1}}-u_{h_{k+1}}\|_1,
\label{Estimate_u_h_{k+1}_Nagative}\\
&&|(f(x,\bar{u}_{h_{k+1}})-f(x,u_{h_{k+1}})-f_u(x,u_{h_{k+1}})(\bar{u}_{h_{k+1}}-u_{h_{k+1}}),v)|
\nonumber\\
&&\quad\quad\quad\quad\quad\quad\quad\quad\quad\quad \lesssim\eta_a(V_H)\|\bar{u}_{h_{k+1}}-u_{h_{k+1}}\|_1\|v\|_1,
\ \ \ \forall v\in V,\label{Nonlinear_Estimate_k+1}
\end{eqnarray}
where $\varepsilon_{h_{k+1}}(u):=\eta_a(V_{h_k})\delta_{h_k}(u)+\|\bar{u}_{h_k}-u_{h_k}\|_0
+|\bar{\lambda}_{h_k}-\lambda_{h_k}|$.
\end{theorem}
%--------------------------------------------------------------------------------------------------------------
\begin{proof}
From (\ref{weak_problem_Discrete}) and (\ref{aux_problem}),
the following estimates hold for any $v_{h_{k+1}}\in V_{h_{k+1}}$
\begin{eqnarray}\label{Estimate_1_Fix}
&&a_{h_k}\big(\bar{u}_{h_{k+1}}-u_{h_k}-\widehat{e}_{h_{k+1}},v_{h_{k+1}}\big)\nonumber\\
&=&a_{h_k}(\bar{u}_{h_{k+1}}-u_{h_k},v_{h_{k+1}})-b(\lambda_{h_k}u_{h_k},v_{h_{k+1}})
+(\nabla u_{h_k},\nabla v_{h_{k+1}})\nonumber\\
&&\ \ \ \ \ +(f(x,u_{h_k}),v_{h_{k+1}})\nonumber\\
&=&(\nabla \bar{u}_{h_{k+1}},\nabla v_{h_{k+1}})+(f_u(x,u_{h_k})(\bar{u}_{h_{k+1}}-u_{h_k}),v_{h_{k+1}})
-b(\lambda_{h_k}u_{h_k},v_{h_{k+1}})\nonumber\\
&&\ \ \ \ \ +(f(x,u_{h_k}),v_{h_{k+1}})\nonumber\\
&=&(f(x,u_{h_k})-f(x,\bar{u}_{h_{k+1}})+f_u(x,u_{h_k})(\bar{u}_{h_{k+1}}-u_{h_k}),v_{h_{k+1}})\nonumber\\
&&\ \ \ \ \ +b(\bar{\lambda}_{h_{k+1}}\bar{u}_{h_{k+1}}-\lambda_{h_k}u_{h_k},v_{h_{k+1}})\nonumber\\
&\lesssim&\big(\|\bar{u}_{h_{k+1}}-u_{h_k}\|_0+|\bar{\lambda}_{h_{k+1}}-\lambda_{h_k}|\big)\|v_{h_{k+1}}\|_0\nonumber\\
&\lesssim&\big(\|\bar{u}_{h_{k+1}}-\bar{u}_{h_k}\|_0+\|\bar{u}_{h_k}-u_{h_k}\|_0
+|\bar{\lambda}_{h_{k+1}}-\bar{\lambda}_{h_k}|+|\bar{\lambda}_{h_k}-\lambda_{h_k}|\big)\|v_{h_{k+1}}\|_1\nonumber\\
&\lesssim&\big(\eta_a(V_{h_k})\delta_{h_k}(u)+\|\bar{u}_{h_k}-u_{h_k}\|_0
+|\bar{\lambda}_{h_k}-\lambda_{h_k}|\big)\|v_{h_{k+1}}\|_1.
\end{eqnarray}
Combing (\ref{Inf_Sup_Condition}) and (\ref{Estimate_1_Fix}), we have the following estimates
\begin{eqnarray}\label{Estimate_u_tilde_u_h_{k+1}}
\|\bar{u}_{h_{k+1}}-u_{h_k}-\widehat{e}_{h_{k+1}}\|_1 &\lesssim
& \sup_{0\neq v_{h_{k+1}}\in V_{h_{k+1}}}\frac{a_{h_k}(\bar{u}_{h_{k+1}}-u_{h_k}
-\widehat{e}_{h_{k+1}},v_{h_{k+1}})}{\|v_{h_{k+1}}\|_1}\nonumber\\
&\lesssim& \eta_a(V_{h_k})\delta_{h_k}(u)+\|\bar{u}_{h_k}-u_{h_k}\|_0
+|\bar{\lambda}_{h_k}-\lambda_{h_k}|.
\end{eqnarray}
Then from (\ref{Estimate_u_tilde_u_h_{k+1}}) and the accuracy
$\|\widehat{e}_{h_{k+1}}-\widetilde{e}_{h_{k+1}}\|_1\lesssim \eta_a(V_{h_k})\delta_{h_k}(u)$,
the following inequality hold
\begin{eqnarray}\label{Error_tilde_u_h_{k+1}_u_final}
\|\bar{u}_{h_{k+1}}-\widetilde{u}_{h_{k+1}}\|_1&\lesssim&\eta_a(V_{h_k})\delta_{h_k}(u)+\|\bar{u}_{h_k}-u_{h_k}\|_0
+|\bar{\lambda}_{h_k}-\lambda_{h_k}|.
\end{eqnarray}

Now we come to estimate the error for the eigenpair solution
$(\lambda_{h_{k+1}},u_{h_{k+1}})$ of problem (\ref{Eigen_Augment_Problem}).
Based on  {\bf Assumptions A1}, {\bf A2} and {\bf C}, and the definition of $V_{H,h_{k+1}}$,
the following estimates hold
\begin{eqnarray}\label{Error_u_u_h_{k+1}}
\|\bar{u}_{h_{k+1}}-u_{h_{k+1}}\|_1\lesssim\inf_{v_{H,h_{k+1}}\in
V_{H,h_{k+1}}}\|\bar{u}_{h_{k+1}}-v_{H,h_{k+1}}\|_1\lesssim
\|\bar{u}_{h_{k+1}}-\widetilde{u}_{h_{k+1}}\|_1,
\end{eqnarray}
and
\begin{eqnarray}
&&|\bar{\lambda}_{h_{k+1}}-\lambda_{h_{k+1}}|+\|\bar{u}_{h_{k+1}}-u_{h_{k+1}}\|_0\lesssim
\eta_a(V_{H,h_{k+1}})\|\bar{u}_{h_{k+1}}-u_{h_{k+1}}\|_1,\label{Error_u_u_h_{k+1}_Negative}\\
&&|(f(x,\bar{u}_{h_{k+1}})-f(x,u_{h_{k+1}})-f_u(x,u_{h_{k+1}})(\bar{u}_{h_{k+1}}
-u_{h_{k+1}}),v)|\lesssim \nonumber\\
&&\quad\quad\quad\quad\quad\quad\eta_a(V_{H,h_{k+1}})\|\bar{u}_{h_{k+1}}-u_{h_{k+1}}\|_1\|v\|_1.
 \ \ \ \forall v\in V,\label{Error_Nonlinear_k+1}
\end{eqnarray}
From (\ref{Property_Eta_h}), (\ref{Error_tilde_u_h_{k+1}_u_final}), (\ref{Error_u_u_h_{k+1}}),
 (\ref{Error_u_u_h_{k+1}_Negative}) and (\ref{Error_Nonlinear_k+1}), the desired results
 (\ref{Estimate_u_u_h_{k+1}}), (\ref{Estimate_u_h_{k+1}_Nagative}) and (\ref{Nonlinear_Estimate_k+1})
 can be obtained and the proof is complete.
\end{proof}

\section{Multigrid scheme for the eigenvalue problem}
In this section, we introduce a type of multigrid correction
scheme based on the {\it One Correction Step} defined in Algorithms
\ref{Correction_Step_Fix} and \ref{Correction_Step_Newton}.
This type of multigrid method can obtain the optimal error
estimate as same as solving the nonlinear eigenvalue problem directly on the finest
finite element space.
%with the same computation work as solving the corresponding
%source problem. %correction method can improve
%the convergence order after each correction step which is different
%from the two-grid method in \cite{XuZhou}.

In order to do multigrid scheme, we define a sequence of triangulations $\mathcal{T}_{h_k}$
of $\Omega$ determined as follows. Suppose $\mathcal{T}_{h_1}$ is produced from $\mathcal{T}_H$
 by regular refinement  and let $\mathcal{T}_{h_k}$ be obtained
from $\mathcal{T}_{h_{k-1}}$ via regular refinement (produce $\beta^d$ subelements) such that
$$h_k\approx\frac{1}{\beta}h_{k-1},\ \ \ \ k=2,\cdots,n.$$
Based on this sequence of meshes, we construct the corresponding linear finite element spaces such that
\begin{eqnarray}\label{FEM_Space_Series}
V_{H}\subseteq V_{h_1}\subset V_{h_2}\subset\cdots\subset V_{h_n},
\end{eqnarray}
and the following relation of approximation errors hold
\begin{eqnarray}\label{Error_k_k_1}
\delta_{h_k}(u)\approx\frac{1}{\beta}\delta_{h_{k-1}}(u),\ \ \ k=2,\cdots,n.
\end{eqnarray}

\begin{algorithm}\label{Multi_Correction}
Eigenvalue Multigrid Scheme
\begin{enumerate}
\item Construct a series of nested finite element
spaces $V_{h_1}, V_{h_2},\cdots,V_{h_n}$ such that
(\ref{FEM_Space_Series}) and (\ref{Error_k_k_1}) hold.
\item Solve the  following nonlinear eigenvalue problem:

Find $(\lambda_{h_1},u_{h_1})\in \mathcal{R}\times V_{h_1}$ such that
$b(u_{h_1},u_{h_1})=1$ and
\begin{eqnarray}\label{Initial_Eigen_Problem}
a(u_{h_1},v_{h_1})&=&\lambda_{h_1}b(u_{h_1},v_{h_1}),\ \ \ \ \forall v_{h_1}\in V_{h_1}.
\end{eqnarray}

\item Do $k=1,\cdots,n-1$\\
Obtain a new eigenpair approximation
$(\lambda_{h_{k+1}},u_{h_{k+1}})\in \mathcal{R}\times V_{h_{k+1}}$
by a correction step defined by Algorithm \ref{Correction_Step_Fix} or \ref{Correction_Step_Newton}
 \begin{eqnarray}
(\lambda_{h_{k+1}},u_{h_{k+1}})=Correction(V_H,\lambda_{h_k},u_{h_k},V_{h_{k+1}}).
\end{eqnarray}
End Do
\end{enumerate}
Finally, we obtain an eigenpair approximation
$(\lambda_{h_n},u_{h_n})\in \mathcal{R}\times V_{h_n}$.
\end{algorithm}
\begin{theorem}\label{Error_MultiGrid_Theorem}
Assume we have conditions of Theorem \ref{Error_Estimate_One_Correction_Theorem_Fix} %{\bf Assumptions A} and {\bf B}
for Algorithm \ref{Multi_Correction} with the
correction step defined by Algorithm \ref{Correction_Step_Fix}, or conditions of
Theorem \ref{Error_Estimate_One_Correction_Theorem}
%{\bf Assumptions A} and {\bf C}
for Algorithm \ref{Multi_Correction} with the
correction step defined by Algorithm \ref{Correction_Step_Newton}.
After implementing Algorithm \ref{Multi_Correction}, the resultant
eigenpair approximation $(\lambda_{h_n},u_{h_n})$ has the following
error estimates
\begin{eqnarray}
\|\bar{u}_{h_n}-u_{h_n}\|_1&\lesssim&\beta^2\eta_a(V_{h_n})\delta_{h_n}(u),\label{Multi_Correction_Err_fun}\\
|\bar{\lambda}_{h_n}-\lambda_{h_n}|+\|\bar{u}_{h_n}-u_{h_n}\|_0&\lesssim&\eta_a(V_{h_n})\delta_{h_n}(u).
\label{Multi_Correction_Err_fun_L2_Eigenvalue}
\end{eqnarray}
under the condition $C\beta\eta_a^2(V_H)<1$ for the constant $C$ hidden in concerned inequalities.
\end{theorem}
\begin{proof}
Here we only give the proof for the case of the
correction step defined by Algorithm \ref{Correction_Step_Fix} and the proof for
Algorithm \ref{Correction_Step_Newton} case can be given similarly.

From the definition of Algorithm \ref{Multi_Correction}, we know that
 $\bar{u}_{h_1}=u_{h_1}$, $\bar{\lambda}_{h_1}=\lambda_{h_1}$.
When $k=2$, from Theorem \ref{Error_Estimate_One_Correction_Theorem_Fix}
and Algorithm \ref{Multi_Correction}, the following
estimates hold
\begin{eqnarray}
\|\bar{u}_{h_2}-u_{h_2}\|_1&\lesssim&\eta_a(V_{h_1})\delta_{h_1}(u),\label{Error_u_h_1_1}\\
|\bar{\lambda}_{h_2}-\lambda_{h_1}|+\|\bar{u}_{h_2}-u_{h_2}\|_0&\lesssim&\eta_a(V_H)\|\bar{u}_{h_2}-u_{h_2}\|_1\nonumber\\
&\leq& \eta_a(V_H)\eta_a(V_{h_1})\delta_{h_1}(u),\label{Error_u_h_1_nagative_1}\\
|(f(x,\bar{u}_{h_2})-f(x,u_{h_2}),v)|&\lesssim& \eta_a(V_H)\|\bar{u}_{h_2}-u_{h_2}\|_1\|v\|_1\nonumber\\
&\lesssim& \eta_a(V_H)\eta_a(V_{h_1})\delta_{h_1}(u)\|v\|_1,\ \ \ \forall v\in V.\label{Nonlinear_Estimate_1_fix}
\end{eqnarray}

Based on Theorem \ref{Error_Estimate_One_Correction_Theorem_Fix}, (\ref{Error_k_k_1}),
 (\ref{Error_u_h_1_1})-(\ref{Nonlinear_Estimate_1_fix}) and recursive argument, the final
 eigenfunction approximation $u_{h_n}$ has the following estimates
\begin{eqnarray}\label{Error_u_h_n_Multi_Correction}
\|\bar{u}_{h_n} - u_{h_n}\|_1 &\lesssim& \eta_{a}(V_{h_{n-1}}) \delta_{h_{n-1}}(u)+
\|\bar{u}_{h_{n-1}} - u_{h_{n-1}}\|_0 +|\bar{\lambda}_{h_{n-1}}-\lambda_{h_{n-1}}|\nonumber\\
&\lesssim& \eta_{a}(V_{h_{n-1}}) \delta_{h_{n-1}}(u)+
\eta_a(V_H)\|\bar{u}_{h_{n-1}} - u_{h_{n-1}}\|_1\nonumber\\
&\lesssim& \eta_{a}(V_{h_{n-1}}) \delta_{h_{n-1}}(u)+
\eta_a(V_H)\eta_{a}(V_{h_{n-2}}) \delta_{h_{n-2}}(u)\nonumber\\
&&\ \ \ \ + \eta_a^2(V_H)\|\bar{u}_{h_{n-2}} - u_{h_{n-2}}\|_1\nonumber\\
&\lesssim& \sum^{n-1}_{k=1}\big(\eta_{a}(V_H)\big)^{n-k-1}\eta_a(V_{h_k})
\delta_{h_k}(u)\nonumber\\
&\lesssim&\Big(\sum^{n-1}_{k=1}\big(\beta^2\eta_{a}(V_H)\big)^{n-k-1}\Big)
\beta^2\eta_a(V_{h_n})\delta_{h_n}(u)\nonumber\\
&\lesssim& \frac{1}{1-\beta^2\eta_a(V_H)}\beta^2\eta_a(V_{h_n})\delta_{h_n}(u)
\lesssim \beta^2\eta_a(V_{h_n})\delta_{h_n}(u).
\end{eqnarray}
This is the desired result (\ref{Multi_Correction_Err_fun}).
Similarly to the proof for Theorem \ref{Error_Estimate_One_Correction_Theorem_Fix}, we can obtain the
result (\ref{Multi_Correction_Err_fun_L2_Eigenvalue}) and the proof is complete.
\end{proof}
%----------------------------------------------------------------------------------------
\begin{remark}
The results (\ref{Multi_Correction_Err_fun}) and (\ref{Multi_Correction_Err_fun_L2_Eigenvalue}) mean that
eigenpair approximation by the multigrid method have the same accuracy both in $L^2(\Omega)$ and $H^1(\Omega)$
as we solve the nonlinear eigenvalue problem directly by the finite element method.
\end{remark}
%----------------------------------------------------------------------------------------
\begin{corollary}
Under the conditions of Theorem \ref{Error_MultiGrid_Theorem}, the eigenpair approximation $(\lambda_{h_n},u_{h_n})$
 by the multigrid method defined by Algorithm \ref{Multi_Correction}
has the following error estimates
\begin{eqnarray}
\|u-u_{h_n}\|_1 &\lesssim&\delta_{h_n}(u),\label{Multi_Correction_Err_fun_Final}\\
|\lambda-\lambda_{h_n}|+\|u-u_{h_n}\|_0 &\lesssim&\eta_a(V_{h_n})\delta_{h_n}(u).
\label{Multi_Correction_Err_fun_L2_Eigenvalue_Final}
\end{eqnarray}
\end{corollary}

\section{Work estimate of eigenvalue multigrid scheme}
%In this section, we come to analyze the computational work for the multigrid
% scheme defined in Algorithm \ref{Algm:Multi_Correction}. Since the linear
% boundary value problem (\ref{Aux_Source_Problem})
% in Algorithm \ref{Algm:One_Step_Correction} is solved by multigrid method,
% the computational work for this part is optimal order.

 In this section, we estimate the computational work
for {\it Eigenvalue Multigrid Scheme} defined by Algorithm \ref{Multi_Correction}.
We will show that Algorithm \ref{Multi_Correction} makes solving eigenvalue problem need almost the
 same work as solving the corresponding linear boundary value problem by the
 multigrid method.

First, we define the dimension of each level linear finite element space as
\begin{eqnarray*}
N_k := {\rm dim}V_{h_k},\ \ \ k=1,\cdots,n.
\end{eqnarray*}
Then we have
\begin{eqnarray}\label{relation_dimension}
N_k \thickapprox\Big(\frac{1}{\beta}\Big)^{d(n-k)}N_n,\ \ \ k=1,\cdots,n.
\end{eqnarray}

The computational work for the second step in Algorithm \ref{Correction_Step_Fix} or
\ref{Correction_Step_Newton} is different
from the linear eigenvalue problems
\cite{LinXie,Xie_Steklov,Xie_Nonconforming,Xie_JCP}.
 In this step, we need to solve a nonlinear eigenvalue problem (\ref{Eigen_Augment_Problem_fix})
 or (\ref{Eigen_Augment_Problem}).
Always, some type of nonlinear iteration method (self-consistent iteration or
 Newton type iteration)
is used to solve this nonlinear eigenvalue problem. In each nonlinear iteration
step, we need to
build the matrix on the finite element space $V_{H,h_k}$ ($k=2,\cdots,n$) which
needs the computational
work $\mathcal{O}(N_k)$.
Fortunately, the matrix building can be carried out by the parallel way easily
in the finite element space since it has no data transfer.

%-------------------------------------------------------------------------------------
\begin{theorem}
 Assume we use $m$ computing-nodes in Algorithm \ref{Multi_Correction},
the nonlinear eigenvalue problem solved in the coarse spaces $V_{H,h_k}$ ($k=1,\cdots, n$)
and $V_{h_1}$ need work $\mathcal{O}(M_H)$ and $\mathcal{O}(M_{h_1})$, respectively, and
the work of multigrid method for solving the boundary value problem in $V_{h_k}$ be $\mathcal{O}(N_k)$
for $k=2,3,\cdots,n$. Let $\varpi$ denote the nonlinear iteration times when we solve
the nonlinear eigenvalue problem (\ref{Eigen_Augment_Problem_fix})
 or (\ref{Eigen_Augment_Problem}).
Then in each computational node, the work involved
in Algorithm \ref{Multi_Correction} has the following estimate
\begin{eqnarray}\label{Computation_Work_Estimate}
{\rm Total\ work}&=&\mathcal{O}\Big(\big(1+\frac{\varpi}{m}\big)N_n
+ M_H\log N_n+M_{h_1}\Big).
\end{eqnarray}
\end{theorem}
%------------------------------------------------------------------------------------
\begin{proof}
Let $W_k$ denote the work in any processor
of the correction step in the $k$-th finite element space $V_{h_k}$.
Then with the correction definition, we have
\begin{eqnarray}\label{work_k}
W_k&=&\mathcal{O}\left(N_k +M_H+\varpi\frac{N_k}{m}\right).
\end{eqnarray}
Iterating (\ref{work_k}) and using the fact (\ref{relation_dimension}), we obtain
\begin{eqnarray}\label{Work_Estimate}
\text{Total work} &=& \sum_{k=1}^nW_k\nonumber =
\mathcal{O}\left(M_{h_1}+\sum_{k=2}^n
\Big(N_k + M_H+\varpi\frac{N_k}{m}\Big)\right)\nonumber\\
&=& \mathcal{O}\Big(\sum_{k=2}^n\Big(1+\frac{\varpi}{m}\Big)N_k
+ (n-1) M_H + M_{h_1}\Big)\nonumber\\
&=& \mathcal{O}\left(\sum_{k=2}^n
\Big(\frac{1}{\beta}\Big)^{d(n-k)}\Big(1+\frac{\varpi}{m}\Big)N_n
+ M_H\log N_n+M_{h_1}\right)\nonumber\\
&=& \mathcal{O}\left(\big(1+\frac{\varpi}{m}\big)N_n
+ M_H\log N_n+M_{h_1}\right).
\end{eqnarray}
This is the desired result and we complete the proof.
\end{proof}
%---------------------------------------------------------------------------------------------------
\begin{remark}
Since we have a good enough initial solution $\widetilde{u}_{h_{k+1}}$
in the second step of Algorithm \ref{Correction_Step_Fix} or \ref{Correction_Step_Newton},
then solving the nonlinear eigenvalue problem (\ref{Eigen_Augment_Problem_fix})
 or (\ref{Eigen_Augment_Problem}) always does not
need many nonlinear iteration  times (always $\varpi\leq 3$).
In this case, the complexity in each computational node will be $\mathcal{O}(N_n)$ provided
$M_H\ll N_n$ and $M_{h_1}\leq N_n$.
\end{remark}

\section{Concluding remarks}
In this paper, we give a type of multigrid scheme
to solve nonlinear eigenvalue problems. The idea here is to use the multilevel correction method
to transform the solution of the nonlinear eigenvalue problem to a series of solutions of the corresponding linear
 boundary value problems with multigrid method and a series of nonlinear eigenvalue problems
 on the coarsest finite element space.
 The proposed multigrid method can be applied to practical nonlinear eigenvalue problems
 \cite{CancesChakirMaday,ChenGongHeYangZhou,ChenGongZhou,ChenHeZhou}.

We can replace the multigrid method by other types of efficient iteration schemes
 such as algebraic multigrid method, the type of preconditioned schemes based on
 the subspace decomposition and subspace corrections (see, e.g., \cite{BrennerScott, Xu}), and the
 domain decomposition method (see, e.g., \cite{ToselliWidlund}).
 Furthermore, the framework here can also be coupled with
 parallel method and the adaptive refinement technique.
   These will be investigated in  our future work.

\end{document}